\documentclass[hyperref]{article}
\usepackage{float} 
\usepackage{amssymb}
\usepackage{amsmath}
\usepackage{mathtools}
\usepackage[thmmarks,amsmath]{ntheorem} 
\usepackage{bm} 
\usepackage[colorlinks,linkcolor=cyan,citecolor=cyan]{hyperref}
\usepackage{enumerate}
\usepackage{extarrows}
\usepackage[hang,flushmargin]{footmisc} 
\usepackage[square,comma,sort&compress,numbers]{natbib} 
\usepackage{mathrsfs} 
\usepackage[font=footnotesize,skip=0pt,textfont=rm,labelfont=rm]{caption,subcaption} 
\usepackage{booktabs} 
\usepackage{tocloft}
\usepackage{authblk}
\usepackage{hyperref}

\newtheorem{thm}{Theorem}[section]
\newtheorem{yl}[thm]{Lemma}
\newtheorem{tl}[thm]{Corollary}
\newtheorem{mt}[thm]{Proposition}

{
	\theoremstyle{nonumberplain}
	\theoremheaderfont{\bfseries}
	\theorembodyfont{\normalfont}
	\theoremsymbol{\mbox{$\Box$}}
	\newtheorem{proof}{Proof}
}

\usepackage{theorem}

{
	\theoremheaderfont{\bfseries}
	\theorembodyfont{\normalfont}
	\newtheorem{remark}{Remark}[section]
}
\usepackage[a4paper]{geometry}                           
\begin{document}
	\title{\bf   The stability for $F$-Yang-Mills functional on $\mathbb{C}P^n$}
	\date{}
	\author{\sffamily Yang Wen$^{1}$\\
		{\sffamily\small $^{1}$Beijing International Center for Mathematical Research, Peking University, Beijing, 100871, China \\2406397066@pku.edu.cn}}

	\maketitle
	{\noindent\small{\bf Abstract:}
		\renewcommand{\thefootnote}{}
		In this paper, we study the critical points of $F$-Yang-Mills functional on $\mathbb{C}P^n$, which are called $F$-Yang-Mills connections, which is a generalization of
		Yang-Mills connections. We prove that if $(2+\frac4n)F''(x)x+(n+1)F'(x)<0$, then the weakly stable $F$-Yang-Mills connection on $\mathbb{C}P^n$ must be flat. Moreover, if $(2+\frac4n)F''(x)x+(n+1)F'(x)=0$, we obtain the structure of curvatures corresponding to weakly stable connections. We also show a gap theorem for $F$-Yang-Mills connections on $\mathbb{C}P^n$. Our approach is inspired by Lawson-Simons' study of Yang-Mills stability on spheres.

	}
	
	\vspace{1ex}
		
		\vspace{1ex}
	\section{Introduction}\ 
	
	Let $(M,g)$ be an n-dimensional compact Riemannian manifold and $E$ be a vector bundle of rank $r$ over $M$ with structure group $G$, where $G$ is a compact Lie group. The classical Yang-Mills functional is given by
	\begin{align*}
		YM(\nabla)=\int_M|R^\triangledown|^2dV,
	\end{align*}
	where $\nabla$ is a connection on $E$ and $R^\triangledown$ is its curvature. We denote $d^\triangledown$ to be the exterior differential operator on $E$ induced by $\nabla$ and $\delta^\triangledown$ is its adjoint operator. The Euler-Lagrange equation of $YM$ is
	\begin{align*}
		\delta^\triangledown R^\triangledown=0
	\end{align*}
	and the critical point of $YM$ is called Yang-Mills connection. 
	
	We are interested in the stability of a Yang-Mills connection. We call $\nabla$ is weakly stable if the second variation of $YM$ is non-negative at $\nabla$, i.e.
	\begin{align*}
		\frac{d^2}{dt^2}YM(\nabla^t)\mid_{t=0}\ge0
	\end{align*}
	for any curve of connections $\nabla^t$ such that $\nabla^0=\nabla$.
	
	This stability research stems from the Lawson–Simons conjecture that there is no stable minimal surface on a compact, simply connected, $\frac14$-pinched Riemannian manifold. By virtue of the profound similarities between minimal surfaces, harmonic maps, and Yang–Mills fields, computational techniques originally developed for minimal surfaces have been successfully transferred to harmonic maps and Yang–Mills theory. Xin \cite{Xin1} proved that there is no weakly stable harmonic maps from $S^n$ for $n\ge3$. Using similar methods, Bourguignon and Lawson \cite{JL} shows that there is no weakly stable Yang-Mills connection on $S^n$ for $n\ge5$. Laquer \cite{Laquer} showed that if the Yang-Mills functional is unstable on a compact irreducible symmetric spaces $M$, then $M$ is isometric to $S^n$ ($n\ge5$), $\mathbb{P}^2(\mathbb{C}ay)$ or $E_6/F_4$. If $M$ is a compact Riemannian manifold isometric immersed in $\mathbb{R}^N$, Kobayashi-Ohnita-Takeuchi \cite{COT} gives the condition of the second fundamental form such that the Yang-Mills functional is unstable. Moreover, Parker \cite{Parker} gives a method to calculate the second variation by using conformal Lagrangian of weight.
	
	The $F$-Yang-Mills functional generalizes the Yang-Mills framework and is studied to gain deeper insights into solutions of Yang-Mills equations. Let $F\in C^2([0,+\infty))$ be a non-negative, increasing function. We define the $F$-Yang-Mills functional to be
	\begin{align}
		\mathscr{A}_F(\nabla)=\int_MF(\frac12|R^\triangledown|^2)dV.
	\end{align}
	When $F(x)=(1+2x)^\alpha$, Hong-Tian-Yin \cite{HTY} introduced this $\alpha$-Yang-Mills functional to investigate the existence of weak solutions for the Yang-Mills flow. Inspired by exponential harmonic maps, by setting $F(x)=e^{2x}$, Fumiaki-Hajime \cite{FH} proposed the exponential Yang-Mills functional.
	
	The Euler-Lagrange equation of $\mathscr{A}_F$ is
	\begin{align}\label{FYM}
		\delta^\triangledown(F'R^\triangledown)=0.
	\end{align}
	The solution of (\ref{FYM}) is called an $F$-Yang-Mills connection. 
	
	Similar to the Yang-Mills functional, a $F$-Yang-Mills connection $\nabla$ is called weakly stable if for any curve of connections $\{\nabla^t\}$ on $E$ such that $\nabla^0=\nabla$, there is
	\begin{align}
		\frac{d^2}{dt^2}\mathscr{A}_F(\nabla^t)\mid_{t=0}\ge0.
	\end{align}
	If $M$ is isometric immersed in $\mathbb{R}^N$, Kurando and Kazuto \cite{KK} give a condition of the second fundamental form such that the $F$-Yang-Mills functional is unstable.
	
	We are interested in the stability of $F$-Yang-Mills functional on $\mathbb{C}P^n$. Inspired by Cheng's study \cite{Cheng} on the stability of Ginzburg-Landau functional on $\mathbb{C}P^n$, we prove the following theorem.
	 \begin{thm}\label{FYM-stability}
	 	\ \\(i)\ If $(2+\frac4n)F''(x)x+(n+1)F'(x)<0$, then the $F$-Yang-Mills connection is unstable on $\mathbb{C}P^n$.
	 	\\(ii)\ Assume $(2+\frac4n)F''(x)x+(n+1)F'(x)=0$ and $\nabla$ is a weakly stable $F$-Yang-Mills connection. Then there exists $\sigma\in\Omega^0(\mathfrak{g}_E)$ such that
	 	\begin{align}
	 		R^\triangledown(X,Y)=g(X,JY)\sigma,
	 	\end{align}
	 	where $g$ is the Fubini-Study metric on $\mathbb{C}P^n$ and $J$ is the almost complex structure.
	 \end{thm}
	 As an example, we have the following corollary.
	 \begin{tl}
	 	(1)\ If $F=x^\alpha$, $0<\alpha<\frac23$, then any $F$-Yang-Mills connection on $\mathbb{C}P^1$ is weakly stable.\\
	 	(2)\ If $F=x^\alpha$, $0<\alpha<\frac14$, then any $F$-Yang-Mills connection on $\mathbb{C}P^2$ is weakly stable.\\
	 \end{tl}
	 We also give a gap theorem for the $F$-Yang-Mills functional.
	 \begin{thm}\label{gap thm}
	 	Assume $\nabla$ is a $F$-Yang-Mills connection on $\mathbb{C}P^n$ $(n\ge2)$ and $F''\ge0$. If $$\|R^\triangledown\|_{L^\infty}\le\frac{(2n-1)\sqrt{2n(2n-1)}}{8(n-1)},$$ then $R^\triangledown\equiv0$.
	 \end{thm}	
	
	\section{Preliminary}
	\subsection{Connections and curvatures on vector bundles}\

	Let $(M,g)$ be an $n$-dimensional compact Riemannian manifold and $D$ be its Levi-Civita connection. Let $E\to M$ be a vector bundle over $M$ with rank $r$ and a compact Lie group $G\subset SO(r)$ be its structure group. We assume $\langle\ ,\ \rangle$ is a Riemannian metric of $E$ compatible with the action of $G$. Let $\mathfrak{g}_E$ be the adjoint bundle of $E$ and $\Omega^p(\mathfrak{g}_E)$ be the collection of $\mathfrak{g}_E$-valued p-forms. $\langle\ ,\ \rangle$ induces the inner product on $\Omega^0(\mathfrak{g}_E)$ such that
	\begin{align*}
		\langle\phi,\psi\rangle=Tr(\phi^T\psi),
	\end{align*}
	where $\phi,\psi\in\Omega^0(\mathfrak{g}_E)$. Then for any $\phi,\psi,\rho\in\Omega^0(\mathfrak{g}_E)$, we have
	\begin{align*}
		\langle[\phi,\psi],\rho\rangle=\langle\phi,[\psi,\rho]\rangle.
	\end{align*}
	The inner product on $\Omega^p(\mathfrak{g}_E)$ is given by
	\begin{align*}
		\langle\phi,\psi\rangle=\frac{1}{p!}\sum_{1\le i_1,...,i_p\le n}\langle\phi(e_{i_1},...,e_{i_p}),\psi(e_{i_1},...,e_{i_p})\rangle,
	\end{align*}
	where $\phi,\psi\in\Omega^p(\mathfrak{g}_E)$ and $\{e_i\}$ is an orthogonal basis of $TM$.	Assume $\nabla:\Omega^0(\mathfrak{g}_E)\to\Omega^1(\mathfrak{g}_E)$ is a connection on $E$ compatible with $\langle\ ,\ \rangle$. Locally, $\nabla$ has the form
	\begin{align*}
		\nabla=d+A,
	\end{align*}
	where $A\in\Omega^1(\mathfrak{g}_E)$. $\nabla$ induces differential operators $d^\triangledown:\Omega^p(\mathfrak{g}_E)\to\Omega^{p+1}(\mathfrak{g}_E)$ and the adjoint operators $\delta^\triangledown:\Omega^p(\mathfrak{g}_E)\to\Omega^{p-1}(\mathfrak{g}_E)$ such that
	\begin{align*}
		&d^\triangledown\phi(X_1,...,X_{p+1})=\sum_{i=1}^{p+1}(-1)^{i+1}\nabla_{X_i}\phi(X_1,...,\hat{X_i},...,X_{p+1})\\
		&\delta^\triangledown\phi(X_1,...,X_{p-1})=-\sum_{i=1}^n\nabla_{e_i}\phi(e_i,X_1,...,X_{p-1})
	\end{align*}
	for any $\phi\in\Omega^p(\mathfrak{g}_E)$ and $X_i\in\Gamma(TM)$.
	
	Let $R^\triangledown\in\Omega^2(\mathfrak{g}_E)$ be the curvature induced by $\nabla$ such that
	\begin{align*}
		R^\triangledown=dA+\frac12A\wedge A,
	\end{align*}
	where
	\begin{align*}
		[\phi,\psi](X,Y)=[\phi(X),\psi(Y)]-[\phi(Y),\psi(X)]
	\end{align*}
	for any $\phi,\psi\in\Omega^1(\mathfrak{g}_E)$ and $X,Y\in\Gamma(TM)$. $R^\triangledown$ induces the curvature on $\Omega^p(\mathfrak{g}_E)$ such that for any $\phi\in\Omega^p(\mathfrak{g}_E)$, we have
	\begin{align*}
		R^\triangledown(X,Y)\phi=\nabla_X\nabla_Y\phi-\nabla_Y\nabla_X\phi-\nabla_{[X,Y]}\phi\in\Omega^p(\mathfrak{g}_E).
	\end{align*}
	By direct calculation, we have
	\begin{align*}
		R^\triangledown(X,Y)\phi(X_1,...,X_p)=[R^\triangledown(X,Y),\phi(X_1,...,X_p)]-\sum_{i=1}^p\phi(X_1,...,R_M(X,Y)X_i,...,X_p).
	\end{align*}
	
	We can define the Laplace-Beltrami operator $\Delta^\triangledown$ by
	\begin{align*}
		\Delta^\triangledown=d^\triangledown\delta^\triangledown+\delta^\triangledown d^\triangledown	
	\end{align*}
	and the rough Laplacian operator by
	\begin{align*}
		\nabla^\ast\nabla=-\sum_j(\nabla_{e_j}\nabla_{e_j}-\nabla_{D_{e_j}e_j}).
	\end{align*}
	For $\psi\in\Omega^1(\mathfrak{g}_E)$ and $\phi\in\Omega^2(\mathfrak{g}_E)$, we define operator $\mathfrak{R}^\triangledown$ as what in \cite{JL}
	\begin{align*}
		&\mathfrak{R}^\triangledown(\psi)(X)=\sum_j[R^\triangledown(e_j,X),\psi(e_j)],\\
		&\mathfrak{R}^\triangledown(\phi)(X,Y)=\sum_j[R^\triangledown(e_j,X),\phi(e_j,Y)]-[R^\triangledown(e_j,Y),\phi(e_j,X)].
	\end{align*}
	Then we have the following Bochner-Weizenb{\"o}ck formula (see \cite{JL}).
	\begin{yl}
		For any $\psi\in\Omega^1(\mathfrak{g}_E)$ and $\phi\in\Omega^2(\mathfrak{g}_E)$, we have
		\begin{align}\label{Bochner}
			&\Delta^\triangledown\psi=\nabla^\ast\nabla\psi+\mathfrak{R}^\triangledown(\psi)+\psi\circ Ric,\\
			&\Delta^\triangledown\phi=\nabla^\ast\nabla\phi+\mathfrak{R}^\triangledown(\phi)+\phi\circ(Ric\wedge I+2R),
		\end{align}
		where
		\begin{itemize}
			\item $R_M$ is the curvature of $TM$,
			\item $Ric :TM\to TM$ is the Ricci transformation defined by
			\[Ric\left( X \right) =\sum_jR(X,e_j)e_j ,\]
			\item $ Ric\wedge Id $ is the extension of the Ricci transformation $Ric$ to $\wedge^2 TM$ given by
			\[ \left( Ric\wedge Id \right)(X,Y)=Ric(X)\wedge Y+X\wedge Ric(Y),\]
			\item For any map $\omega:\wedge^2 TM\to \wedge^2 TM$, the composite map $\varphi\circ \omega :\wedge^2 TM \to \Omega^0 \left( \mathfrak{g}_E \right)  $ is defined by
			\[\left( \varphi\circ \omega \right)({ X,Y })=\frac{1}{2}\sum_{j=1}^{n} \varphi({ e_j, \omega({ X,Y })e_j }) .\]
		\end{itemize}
	\end{yl}
	
	Note that for any $B\in\Omega^1(\mathfrak{g}_E)$, we have
	\begin{align*}
		R^{\triangledown+tB}=R^\triangledown+td^\triangledown B+\frac12t^2[B\wedge B].
	\end{align*}
	\subsection{Curvatures and killing vector fields on $\mathbb{C}P^n$}
	
	Now we assume $M$ is the complex projection space $\mathbb{C}P^n$ with Fubini-Study metric $g$ and $J$ is the almost complex structure. Assume $U_0=\{[z_0:...:z_n]\mid z_0\ne0\}\subset \mathbb{C}P^n$ and let $(z_1,...,z_n):=[1:z_1:...:z_n]$ be the local coordinate system on $U_0$. The Fubini-Study metric on $U_0$ is
	\begin{align}
		g_{i\bar j}=\frac{\delta_{ij}}{1+|z|^2}-\frac{\bar z_iz_j}{(1+|z|^2)^2},
	\end{align}
	where $|z|^2=\sum_{i=1}^n|z_i|^2$.
	For simplicity, we denote $Z_i=\frac{\partial}{\partial z_i}$ and $\bar Z_i=\frac{\partial}{\partial \bar z_i}$. Due to the symmetry of $\mathbb{C}P^n$, we only need to consider the curvature at $z_0=[1:0:,...,:0]$. For $1\le\alpha\le n$, let
	\begin{equation}\label{ealpha}
		\begin{split}
			e_\alpha&=\frac{1}{\sqrt2}(Z_\alpha+\bar Z_\alpha),\\
			e_{n+\alpha}&=\frac{i}{\sqrt2}(Z_\alpha-\bar Z_{\alpha}).
		\end{split}
	\end{equation}
	Then $\{e_1,...,e_{2n}\}$ is an orthogonal basis of $T_{z_0}\mathbb{C}P^n$ and $Je_\alpha=e_{n+\alpha}$, $Je_{n+\alpha}=-e_\alpha$. For $1\le\alpha,\beta,\gamma\le n$, the curvature tensors at $z_0$ are
	\begin{equation}\label{curvature}
		\begin{split}
			&R(e_\alpha,e_\beta)e_\gamma=\frac12\delta_{\beta\gamma}e_\alpha-\frac12\delta_{\alpha\gamma}e_\beta,\\
			&R(e_\alpha,e_\beta)e_{n+\gamma}=\frac12\delta_{\beta\gamma}e_{n+\alpha}-\frac12\delta_{\alpha\gamma}e_{n+\beta},\\
			&R(e_{n+\alpha},e_\beta)e_\gamma=\delta_{\alpha\beta}e_{n+\gamma}+\frac12\delta_{\beta\gamma}e_{n+\alpha}+\frac12\delta_{\alpha\gamma}e_{n+\beta},\\
			&R(e_{n+\alpha},e_\beta)e_{n+\gamma}=-\delta_{\alpha\beta}e_\gamma-\frac12\delta_{\beta\gamma}e_\alpha-\frac12\delta_{\alpha\gamma}e_\beta,\\
			&R(e_{n+\alpha},e_{n+\beta})e_\gamma=\frac12\delta_{\beta\gamma}e_\alpha-\frac12\delta_{\alpha\gamma}e_\beta,\\
			&R(e_{n+\alpha},e_{n+\beta})e_{n+\gamma}=\frac12\delta_{\beta\gamma}e_{n+\alpha}-\frac12\delta_{\alpha\gamma}e_{n+\beta}.
		\end{split}
	\end{equation}
	Then we have
	\begin{equation}\label{Ricci}
		\begin{split}
			&(i)\ Ric(X)=(n+1)X,\\			
			&(ii)\ R^\triangledown\circ 2R(X,Y)=-R^\triangledown(X,Y)-R^\triangledown(JX,JY)-\sum_{j=1}^{2n}R^\triangledown(e_j,Je_j)g(JX,Y),\\
			&(iii)\ \sum_{j=1}^{2n}R(JX,e_j)Je_j=-(n+1)X.
		\end{split}
	\end{equation}
	A standard technique for stability computations on compact manifolds employs the construction of special vector fields on the manifold, which are then used to generate variation directions. Killing fields on $\mathbb{C}P^n$ play a crucial role in constructing such variations. Recall the theorems for killing vector fields on $\mathbb{C}P^n$ (see \cite{Jost}).
	\begin{yl}\label{Killing}
		Let $\mathcal{K}$ be the collection of Killing vector fields on $\mathbb{C}P^n$, then\\
		(i)\ For any $V\in\mathcal{K}$ and vector fields $X,Y$, we have $D^2_{X,Y}V=R(X,V)Y$, where $D^2_{X,Y}V=D_XD_YV-D_{D_XY}V$. Especially, $Ric(V)=D^\ast DV$.\\
		(ii)\ $\mathcal{K}$ is a finite dimensional $\mathbb{R}$-vector space.\\
		(iii)\ For any $x\in \mathbb{C}P^n$, we have
		\begin{align*}
			\mathcal{K}\cong\mathfrak{f}_x\oplus\mathfrak{p}_x,
		\end{align*}
		where
		\begin{align*}
			\mathfrak{f}_x=\{V\in\mathcal{K}\mid V(x)=0\}
		\end{align*}
		and
		\begin{align*}
			\mathfrak{p}_x=\{V\in\mathcal{K}\mid DV(x)=0\}\cong T_x\mathbb{C}P^n.
		\end{align*}
	\end{yl}
	In fact, Kobayashi and Nomizu \cite{SK} give an explicit form of the killing vector fields on $\mathbb{C}P^n$.
	\begin{mt}\label{Killing&matrix}
		For any $A\in\mathfrak{su}(n+1)$ and $t\in\mathbb{R}$, we define $\varphi_A^t:\mathbb{C}P^n\to\mathbb{C}P^n$ such that
		\begin{align*}
			\varphi_A^t([z])=[exp(tA)z].
		\end{align*}
		Then we have
		\begin{align*}
			\mathcal{K}=\{V_A=\frac{\partial \varphi_A^t}{\partial t}\mid_{t=0}\mid A\in\mathfrak{su}(n+1) \}.
		\end{align*}
	\end{mt}
	We define the inner product $\langle\ ,\ \rangle_{\mathcal{K}}$ on $\mathcal{K}$ such that
	\begin{align*}
		\langle V_A,V_B\rangle_\mathcal{K}:=tr(\bar A^T B).
	\end{align*}
	For $0\le k<l\le n$ and $1\le t\le n$, assume
	\begin{equation}\label{ABC}
		A_{kl}=\frac{1}{\sqrt{2}}(E_{kl}-E_{lk}),\ B_{kl}=\frac{\sqrt{-1}}{\sqrt2}(E_{kl}+E_{lk}),\ C_t=\frac{\sqrt{-1}}{\sqrt{t(t+1)}}(\sum_{s=0}^{t-1}E_{ss}-tE_{tt}).
	\end{equation}
	Then $\{V_{A_{kl}},V_{B_{kl}},V_{C_t}\mid 0\le k<l\le n,1\le t\le n\}$ is an orthogonal basis of $\mathcal{K}$. Assume $A=(a_{ij})_{0\le i,j\le n}\in\mathfrak{su}(n+1)$. On $U_0$, we have
	\begin{align*}
		V_A=\sum_{j=1}^n(a_{j0}+a_{ji}z_i-a_{00}z_j-a_{0i}z_iz_j)Z_j+(\bar a_{j0}+\bar a_{ji}\bar z_i-\bar a_{00}\bar z_j-\bar a_{0i}\bar z_i\bar z_j)\bar Z_j.
	\end{align*}
	At $z_0=[1:0:...:0]$, we have
	\begin{equation}\label{V+DV}
		\begin{split}
			&JV_A(z_0)=\sqrt{-1}\sum_{j=1}^n(a_{j0}Z_j+a_{0j}\bar Z_j),\\
			&D_{Z_k}JV_A(z_0)=\sqrt{-1}\sum_{j=1}^n(a_{jk}Z_j-a_{00}Z_k),\\
			&D_{\bar Z_l}JV_A(z_0)=\sqrt{-1}\sum_{j=1}^n(_{lj}\bar Z_j-a_{00}\bar Z_l).
		\end{split}
	\end{equation}
	Thus
	\begin{align*}
		\mathfrak{p}_{z_0}=span_{\mathbb{R}}\{V_{A_{0k}},V_{B_{0k}}\mid 1\le k\le n\}.
	\end{align*}
	\section{The stability and a gap theorem of $F$-Yang-Mills functional on $\mathbb{C}P^n$}
	\subsection{The second variation of $F$-Yang-Mills functional}\ 
	
	Assume $\nabla$ is a $F$-Yang-Mills connection on $\mathbb{C}P^n$ and $\{\nabla^t\}$ is a family of connections on $E$ compatible with $\langle\ ,\ \rangle$ such that $\nabla^0=\nabla$ and $\frac{d}{dt}\nabla^t\mid_{t=0}=B$. The second variation of $\mathscr{A}_F$ is
	\begin{align*}
		\frac{d^2}{dt^2}\mathscr{A}_F(\nabla^t)\mid_{t=0}:=\mathscr{L}_F^\triangledown(B),
	\end{align*}
	where
	\begin{align}
		\mathscr{L}_F^\triangledown(B)=\int_{\mathbb{C}P^n}F''\langle R^\triangledown,d^\triangledown B\rangle^2+\langle \delta^\triangledown(F'd^\triangledown B),B\rangle+F'\langle\mathfrak{R}^\triangledown(B),B\rangle dV.
	\end{align}
	For simplicity, we define
	\begin{align*}             X_0=grad(\frac12|R^\triangledown|^2)=\sum_{j=1}^{2n}\langle\nabla_{e_j}R^\triangledown,R^\triangledown\rangle e_j.
	\end{align*}
	Then the $F$-Yang-Mills equation is equal to
	\begin{align}\label{FYM equal}
		F'\delta^\triangledown R^\triangledown=F''i_{X_0}R^\triangledown,
	\end{align}
	where $i_XR^\triangledown\in\Omega^1(\mathfrak{g}_E)$ is the contraction about $X$.
	
	For any $V\in\mathcal{K}$, we demonstrate that using $i_VR^\triangledown$ as the variation direction yields zero second variation. Thus, we employ $JV$ to construct the variation direction, which is a vector field everywhere orthogonal to $V$. This selection method is inspired by the construction of special vector fields on spheres (see Remark \ref{rmk:eigenfunction, projection}), yet here we adopt the perspective where the complex structure acts on Killing vector fields, primarily because Lemma \ref{Killing} and Proposition \ref{Killing&matrix} significantly facilitates subsequent computations.
	\begin{yl}\label{2nd var}
	    For any $V\in\mathcal{K}$, let $B_V=i_{JV}R^\triangledown$. Then we have
	    \begin{align}
	    	\mathscr{L}_F^\triangledown(B_V)=\int_{\mathbb{C}P^n}J_V^1(x)+J_V^2(x)+J_V^3(x)+J_V^4(x)dV,
	    \end{align}
	    where
	    \begin{equation}
	    	\begin{split}
	    		J_V^1=&F''(\langle R^\triangledown,\nabla_{JV}R^\triangledown\rangle^2-\langle i_{X_0}\nabla_{JV}R^\triangledown,i_{JV}R^\triangledown\rangle),\\
	    		J_V^2=&F''(2\sum_{i,j}\langle R^\triangledown,\nabla_{JV}R^\triangledown\rangle\langle R^\triangledown(e_i,e_j),R^\triangledown(D_{e_i}JV,e_j)\rangle-\langle i_{D_{X_0}JV}R^\triangledown,i_{JV}R^\triangledown\rangle)\\
	    		&-F'\sum_{i,j}\langle\nabla_{e_i}R^\triangledown(D_{e_i}JV,e_j)+\nabla_{D_{e_i}JV}R^\triangledown(e_i,e_j),R^\triangledown(JV,e_j)\rangle,\\
	    		J_V^3=&F'\langle\nabla_{JV}\delta^\triangledown R^\triangledown,i_{JV}R^\triangledown\rangle,\\
	    		J_V^4=&F''(\sum_{i,j}\langle R^\triangledown(e_i,e_j),R^\triangledown(D_{e_i}JV,e_j)\rangle)^2+F'\langle i_{D^\ast DJV}R^\triangledown-i_{Ric(JV)}R^\triangledown,i_{JV}R^\triangledown\rangle\\
	    		&+F'\sum_{i,j}\langle R^\triangledown (D^2_{e_i,e_j}JV,e_i)+R^\triangledown(e_i,R(e_i,JV)e_j),R^\triangledown(JV,e_j)\rangle.\\
	    	\end{split}
	    \end{equation}
	\end{yl}
	\begin{proof}
	For any $x\in\mathbb{C}P^n$, let $\{e_1,...,e_{2n}\}$ be an orthogonal basis of $T\mathbb{C}P^n$ such that $De_i(x)=0$. We have
	\begin{equation}\label{dJVR}
		\begin{split}
		d^\triangledown i_{JV}R^\triangledown(e_i,e_j)&=\nabla_{e_i}i_{JV}R^\triangledown(e_j)-\nabla_{e_j}i_{JV}R^\triangledown(e_i)\\
		&=\nabla_{e_i}(R^\triangledown(JV,e_j))-\nabla_{e_j}(R^\triangledown(JV,e_i))\\
		&=\nabla_{e_i}R^\triangledown(JV,e_j)+R^\triangledown(D_{e_i}JV,e_j)-\nabla_{e_j}R^\triangledown(JV,e_i)-R^\triangledown(D_{e_j}JV,e_i)\\
		&=\nabla_{JV}R^\triangledown(e_i,e_j)+R^\triangledown(D_{e_i}JV,e_j)-R^\triangledown(D_{e_j}JV,e_i),
		\end{split}
	\end{equation}
	where we use $d^\triangledown R^\triangledown=0$. Thus
	\begin{align*}
		&F''\langle R^\triangledown,d^\triangledown i_{JV}R^\triangledown\rangle^2\\
		=&F''(\langle R^\triangledown,\nabla_{JV}R^\triangledown\rangle^2+\sum_{i,j=1}^{2n}(\langle R^\triangledown(e_i,e_j),R^\triangledown(D_{e_i}JV,e_j)\rangle^2+2\langle R^\triangledown,\nabla_{JV}R^\triangledown\rangle\langle R^\triangledown(e_i,e_j),R^\triangledown(D_{e_i}JV,e_j)\rangle)).
	\end{align*}
	Note that (\ref{dJVR}) is independent to the choice of $\{e_i\}$, we have
	\begin{align*}
		&\delta^\triangledown(F'd^\triangledown i_{JV}R^\triangledown)(e_j)\\
		=&-\sum_{i=1}^{2n}\nabla_{e_i}(F'd^\triangledown i_{JV}R^\triangledown)(e_i,e_j)\\
		=&-\sum_{i=1}^{2n}F''\langle\nabla_{e_i}R^\triangledown,R^\triangledown\rangle d^\triangledown i_{JV}R^\triangledown(e_i,e_j)-\sum_{i=1}^{2n}F'\nabla_{e_i}(d^\triangledown i_{JV}R^\triangledown(e_i,e_j))\\
		=&-F''d^\triangledown i_{JV}R^\triangledown(X_0,e_j)-\sum_{i=1}^{2n}F'\nabla_{e_i}(\nabla_{JV}R^\triangledown(e_i,e_j)+R^\triangledown(D_{e_i}JV,e_j)-R^\triangledown(D_{e_j}JV,e_i))\\
		=&-F''(\nabla_{JV}R^\triangledown(X_0,e_j)+R^\triangledown(D_{X_0}JV,e_j)-R^\triangledown(D_{e_j}JV,X_0))\\
		&+F'(R^\triangledown(D^\ast DJV,e_j)+\delta^\triangledown R^\triangledown(D_{e_j}JV)-\sum_{i=1}^{2n}(\nabla_{e_i}\nabla_{JV}R^\triangledown(e_i,e_j)+\nabla_{e_i}R^\triangledown(D_{e_i}JV,e_j)-R^\triangledown(D^2_{e_i,e_j}JV,e_i))),
	\end{align*}
	where by (\ref{FYM equal}) we have
	\begin{align*}
		F''R^\triangledown(D_{e_j}JV,X_0)+F'\delta^\triangledown R^\triangledown(D_{e_j}JV)=0.
	\end{align*}
	Finally, at $x$ we have
	\begin{align*}
		&\sum_i\nabla_{e_i}\nabla_{JV}R^\triangledown(e_i,e_j)\\
		=&\sum_i([R^\triangledown(e_i,JV),R^\triangledown(e_i,e_j)]-R^\triangledown(R(e_i,JV)e_i,e_j)-R^\triangledown(e_i,R(e_i,JV)e_j)+\nabla_{JV}\nabla_{e_i}R^\triangledown(e_i,e_j)\\
		&+\nabla_{D_{e_i}JV}R^\triangledown(e_i,e_j))\\
		=&\mathfrak{R}^\triangledown(i_{JV}R^\triangledown)(e_j)+R^\triangledown(Ric(JV),e_j)-\nabla_{JV}\delta^\triangledown R^\triangledown(e_j)+\sum_i(\nabla_{D_{e_i}JV}R^\triangledown(e_i,e_j)-R^\triangledown(e_i,R(e_i,JV)e_j)).
	\end{align*}
	We have completed the proof through the above calculations.
	\end{proof}
	\begin{yl}\label{sum J123}
		Let $q=dim_{\mathbb{R}}(\mathcal{K})=n^2+2n$. Assume $\{V_1,...,V_q\}$ be an orthogonal basis of $\mathcal{K}$. Then for any $x\in\mathbb{C}P^n$, we have
		\begin{align*}
		   (i)\ \sum_{k=1}^q&J_{V_k}^1(x)=\sum_{k=1}^qJ_{V_k}^2(x)=0,\\
		   (ii)\ \sum_{k=1}^q&\int_{\mathbb{C}P^n}J_{V_k}^3(x)dV=0.
		\end{align*}
	\end{yl}
	\begin{proof}
		Since $\sum_{k=1}^qJ_{V_k}^\mu(x)$, $\mu\in\{1,2,3,4\}$ is independent to the choice of $\{V_1,...,V_q\}$, we assume $V_1,...,V_{2n}\in\mathfrak{p}_x$ and $V_{2n+1},...,V_q\in\mathfrak{f}_x$. Then $\{JV_1(x),...,JV_{2n}(x)\}$ is an orthogonal basis of $T_x\mathbb{C}P^n$ and $JV_{2n+1}(x)=...=JV_q(x)=0$. Thus we have
	\begin{align*}
		\sum_{k=1}^q\langle\nabla_{JV_k}\delta^\triangledown R^\triangledown,i_{JV_k}R^\triangledown\rangle=\sum_{i=1}^{2n}\langle\nabla_{e_i}\delta^\triangledown R^\triangledown,i_{e_i}R^\triangledown\rangle
	\end{align*}
	at $x$. By the equation (\ref{FYM}) we have
	\begin{align*}
		\sum_{k=1}^q\int_{\mathbb{C}P^n}J_{V_k}^3(x)dV=&\sum_{k=1}^q\int_{\mathbb{C}P^n}F'\langle\nabla_{JV_k}\delta^\triangledown R^\triangledown,i_{JV_k}R^\triangledown\rangle dV\\
		=&\int_{\mathbb{C}P^n}F'\sum_{i,j}\langle\nabla_{e_i}\delta^\triangledown R^\triangledown(e_j),R^\triangledown(e_i,e_j)\rangle dV\\
		=&\frac12\int_{\mathbb{C}P^n}F'\sum_{i,j}\langle\nabla_{e_i}\delta^\triangledown R^\triangledown(e_j)-\nabla_{e_j}\delta^\triangledown R^\triangledown(e_i),R^\triangledown(e_i,e_j)\rangle dV\\
		=&\int_{\mathbb{C}P^n}F'\langle d^\triangledown \delta^\triangledown R^\triangledown,R^\triangledown\rangle dV\\
		=&\int_{\mathbb{C}P^n}\langle\delta^\triangledown R^\triangledown,\delta^\triangledown(F'R^\triangledown)\rangle dV\\
		=&0.
	\end{align*}
	By the definition of $X_0$, at $x$ we have
	\begin{align*}
		\sum_{k=1}^q\langle R^\triangledown,\nabla_{JV_k}R^\triangledown\rangle^2=\sum_{j=1}^{2n}\langle R^\triangledown,\nabla_{e_j}R^\triangledown\rangle^2=\langle R^\triangledown,\nabla_{X_0}R^\triangledown\rangle
	\end{align*}
	and
	\begin{align*}
		\sum_{k=1}^q\langle i_{X_0}\nabla_{JV_k}R^\triangledown,i_{JV_k}R^\triangledown\rangle&=\sum_{i,j=1}^{2n}\langle\nabla_{e_i}R^\triangledown(X_0,e_j),R^\triangledown(e_i,e_j)\rangle.
	\end{align*}
	By Bianchi identity, we have
	\begin{align*}
		\sum_{i,j=1}^{2n}\langle\nabla_{e_i}R^\triangledown(X_0,e_j),R^\triangledown(e_i,e_j)\rangle=&-\sum_{i,j=1}^{2n}\langle\nabla_{X_0}R^\triangledown(e_j,e_i)+\nabla_{e_j}R^\triangledown(e_i,X_0),R^\triangledown(e_i,e_j)\rangle\\
		=&-\sum_{i,j=1}^{2n}\langle\nabla_{e_i}R^\triangledown(X_0,e_j),R^\triangledown(e_i,e_j)\rangle+2\langle\nabla_{X_0}R^\triangledown,R^\triangledown\rangle.
	\end{align*}
	Then we get
	\begin{align*}
		\sum_{k=1}^q\langle i_{X_0}\nabla_{JV_k}R^\triangledown,i_{JV_k}R^\triangledown\rangle=\langle\nabla_{X_0}R^\triangledown,R^\triangledown\rangle.
	\end{align*}
	Based on the above calculations, we obtain
	\begin{align*}
		\sum_{k=1}^qJ_{V_k}^1(x)=0.
	\end{align*}
	Finally, it is easy to see that $$\sum_{k=1}^qJ_{V_k}^2(x)=0$$ since the choice of $V_k$ implies either $JV_k(x)=0$ or $DJV_k(x)=0$.
	\end{proof}
	We use three lemmas to calculate $\sum_kJ_{V_k}^4$. We prove the following lemma by the curvature properties of $\mathbb{C}P^n$.
	\begin{yl}\label{sum J4-1}
		(i)\ For any $V\in\mathcal{K}$, we have
		\begin{align*}
			D^\ast DJV=Ric(JV),
		\end{align*}
		and thus
		\begin{align*}
			F'\langle i_{D^\ast DJV}R^\triangledown-i_{Ric(JV)}R^\triangledown,i_{JV}R^\triangledown\rangle=0.
		\end{align*}
		(ii)\ Assume $\{V_1,...,V_q\}$ be an orthogonal basis of $\mathcal{K}$, then for any $x\in\mathbb{C}P^n$, we have
		\begin{align*}
			&\sum_{k=1}^q\sum_{i,j=1}^{2n}\langle R^\triangledown(D^2_{e_i,e_j}JV_k,e_i)+R^\triangledown(e_i,R(e_i,JV_k)e_j),R^\triangledown(JV_k,e_j)\rangle\\
			=&2|R^\triangledown|^2+\sum_{i,j=1}^{2n}\langle R^\triangledown(e_i,e_j),R^\triangledown(Je_i,Je_j)\rangle+|\sum_{j=1}^{2n}R^\triangledown(e_j,Je_j)|^2.
		\end{align*}
	\end{yl}
	\begin{proof}
		(i) By (\ref{Ricci}) and lemma \ref{Killing}, we have 
		\begin{align*}
			D^\ast DJV=JD^\ast DV=JRic(V)=(n+1)JV=Ric(JV).
		\end{align*}
		(ii) For any $x\in\mathbb{C}P^n$, we assume $V_1,...,V_{2n}\in\mathfrak{p}_x$ and $V_{2n+1},...,V_q\in\mathfrak{f}_x$. By lemma \ref{Killing}, for any $V\in\mathcal{K}$ we have
		\begin{align*}
			D^2_{e_i,e_j}JV=JD^2_{e_i,e_j}V=JR(e_i,V)e_j=R(e_i,V)Je_j.
		\end{align*}
		Thus
	    \begin{align*}
	    	&\sum_{k=1}^q\sum_{i,j=1}^{2n}\langle R^\triangledown(D^2_{e_i,e_j}JV_k,e_i)+R^\triangledown(e_i,R(e_i,JV_k)e_j),R^\triangledown(JV_k,e_j)\rangle\\
	    	=&-\sum_{i,j,k=1}^{2n}\langle R^\triangledown(R(e_i,Je_k)Je_j,e_i)+R^\triangledown(R(e_i,e_k)e_j,e_i),R^\triangledown(e_k,e_j)\rangle.
	    \end{align*}
	    By Bianchi identity, we have
	    \begin{align*}
	    	-\sum_{i,j,k=1}^{2n}\langle R^\triangledown(R(e_i,Je_k)Je_j,e_i),R^\triangledown(e_k,e_j)\rangle=-\frac12\sum_{j,k=1}^{2n}\langle R^\triangledown\circ2R(Je_k,Je_j),R^\triangledown(e_k,e_j)\rangle
	    \end{align*}
	    and
	    \begin{align*}
	    	-\sum_{i,j,k=1}^{2n}\langle R^\triangledown(R(e_i,e_k)e_j,e_i),R^\triangledown(e_k,e_j)\rangle=-\langle R^\triangledown\circ (2R),R^\triangledown\rangle.
	    \end{align*}
	    By substituting (\ref{Ricci}), we have
	    \begin{align*}
	    	&-\frac12\sum_{j,k=1}^{2n}\langle R^\triangledown\circ2R(Je_k,Je_j),R^\triangledown(e_k,e_j)\rangle\\
	    	=&\frac12\sum_{j,k=1}^{2n}\langle R^\triangledown(Je_k,Je_j)+R^\triangledown(e_k,e_j)-\sum_{i=1}^{2n}g(e_k,Je_j)R^\triangledown(e_i,Je_i),R^\triangledown(e_k,e_j)\rangle\\
	    	=&|R^\triangledown|^2+\frac12\sum_{j,k=1}^{2n}\langle R^\triangledown(e_j,e_k),R^\triangledown(Je_j,Je_k)\rangle+\frac12|\sum_{i=1}^{2n}R^\triangledown(e_i,Je_i)|^2
	    \end{align*}
	    and
	    \begin{align*}
	    	&-\langle R^\triangledown\circ (2R),R^\triangledown\rangle\\
	    	=&\frac12\sum_{j,k=1}^{2n}\langle R^\triangledown(e_j,e_k)+R^\triangledown(Je_j,Je_k)+\sum_{i=1}^{2n}g(Je_j,e_k)R^\triangledown(e_i,Je_i),R^\triangledown(e_j,e_k)\rangle\\
	    	=&|R^\triangledown|^2+\frac12\sum_{j,k=1}^{2n}\langle R^\triangledown(e_j,e_k),R^\triangledown(Je_j,Je_k)\rangle+\frac12|\sum_{i=1}^{2n}R^\triangledown(e_i,Je_i)|^2.
	    \end{align*}
	    By adding the above two equations, we can complete the proof of the lemma.
	\end{proof}
	The remaining term depends on the specific form of $DV$. Since the symmetry of $\mathbb{C}P^n$, we calculate the term at $z_0=[1:0,...,0]$. Since $$F''(\sum_{i,j=1}^{2n}\langle R^\triangledown(e_i,e_j),R^\triangledown(D_{e_i}JV,e_j)\rangle)^2=0$$
	for $V\in\mathfrak{p}_{z_0}$, we only need to consider $V\in\mathfrak{f}_{z_0}=span\{V_{A_{kl}},V_{B_{kl}},V_{C_t}\mid1\le k<l\le n,1\le t\le n\}$. By combining (\ref{ealpha}) and (\ref{V+DV}), we obtain
	\begin{equation}\label{DJV}
		\begin{split}
			&D_{e_\alpha}JV_{A_{kl}}=\frac1{\sqrt2}(\delta_{\alpha l}e_{n+k}-\delta_{\alpha k}e_{n+l}),\\
			&D_{e_{n+\alpha}}JV_{A_{kl}}=\frac1{\sqrt2}(\delta_{\alpha k}e_l-\delta_{\alpha l}e_k),\\
			&D_{e_\alpha}JV_{B_{kl}}=-\frac1{\sqrt2}(\delta_{\alpha l}e_k+\delta_{\alpha k}e_l),\\
			&D_{e_{n+\alpha}}JV_{B_{kl}}=-\frac1{\sqrt2}(\delta_{\alpha l}e_{n+k}+\delta_{\alpha k}e_{n+l}),\\
			&D_{e_\alpha}JV_{C_t}=\left \{\begin{array}{cl}0,&\alpha\le t-1,\\
				\sqrt{\frac{t+1}k}e_t,&\alpha=t,\\
				\frac1{\sqrt{t(t+1)}}e_\alpha,&\alpha\ge t+1,\end{array} \right.\\
			&D_{e_{n+\alpha}}JV_{C_t}=\left \{\begin{array}{cl}0,&\alpha\le t-1,\\
				\sqrt{\frac{t+1}t}e_{n+t},&\alpha=t,\\
				\frac1{\sqrt{t(t+1)}}e_{n+\alpha},&\alpha\ge t+1,\end{array} \right.
		\end{split}
	\end{equation}
	at $z_0$ for any $1\le k<l\le n$ and $1\le t,\alpha\le n$. Now we can prove the following lemma.
	\begin{yl}\label{sum J4-2}
		For any $1\le t\le n$ and $1\le k<l\le n$, let $A_{kl}, B_{kl}, C_t$ be matrices defined in (\ref{ABC}). The corresponding Killing vector fields $V_{A_{kl}},V_{B_{kl}},V_{C_t}$ is defined in Proposition \ref{Killing&matrix}. Then we have
		\begin{align*}
			(i)\ &\sum_{1\le k<l\le n}(\sum_{i,j=1}^{2n}\langle R^\triangledown(e_i,e_j),R^\triangledown(D_{e_i}JV_{A_{kl}},e_j)\rangle)^2=\sum_{\alpha,\beta=1}^n(\langle i_{e_\alpha}R^\triangledown,i_{e_{n+\beta}}R^\triangledown\rangle-\langle i_{e_{n+\alpha}}R^\triangledown,i_{e_\beta}R^\triangledown\rangle)^2,\\
			(ii)\ &\sum_{1\le k<l\le n}(\sum_{i,j=1}^{2n}\langle R^\triangledown(e_i,e_j),R^\triangledown(D_{e_i}JV_{B_{kl}},e_j)\rangle)^2=\sum_{\alpha,\beta=1}^n(\langle i_{e_\alpha}R^\triangledown,i_{e_\beta}R^\triangledown\rangle+\langle i_{e_{n+\alpha}}R^\triangledown,i_{e_{n+\beta}}R^\triangledown\rangle)^2\\
			&-\frac12\sum_{i=1}^{2n}(|i_{e_i}R^\triangledown|^2+|i_{Je_i}R^\triangledown|^2)^2,\\
			(iii)\  &\sum_{t=1}^n(\sum_{i,j=1}^{2n}\langle R^\triangledown(e_i,e_j),R^\triangledown(D_{e_i}JV_{C_t},e_j)\rangle)^2=\frac12\sum_{i=1}^{2n}(|i_{e_i}R^\triangledown|^2+|i_{Je_i}R^\triangledown|^2)^2+4|R^\triangledown|^4
		\end{align*}
		at $z_0$, where $\{e_1,...,e_{2n}\}$ is an orthogonal basis of $T_{z_0}\mathbb{C}P^n$ defined in (\ref{ealpha}).
	\end{yl}
	\begin{proof}
		(i)\ For any $1\le k<l\le n$, by (\ref{DJV}) we have
		\begin{align*}
			&\sum_{i,j=1}^{2n}\langle R^\triangledown(e_i,e_j),R^\triangledown(D_{e_i}JV_{A_{kl}},e_j)\rangle\\
			=&\sum_{\alpha=1}^n\sum_{j=1}^{2n}(\langle R^\triangledown(e_\alpha,e_j),R^\triangledown(D_{e_\alpha}JV_{A_{kl}},e_j)\rangle\langle R^\triangledown(e_{n+\alpha},e_j),R^\triangledown(D_{e_{n+\alpha}}JV_{A_{kl}},e_j)\rangle)\\
			=&\frac1{\sqrt2}\sum_{\alpha=1}^n\sum_{j=1}^{2n}(\langle R^\triangledown(e_\alpha,e_j),R^\triangledown(\delta_{\alpha l}e_{n+k}-\delta_{\alpha k}e_{n+l},e_j)\rangle+\langle R^\triangledown(e_{n+\alpha},e_j),R^\triangledown(\delta_{\alpha k}e_l-\delta_{\alpha l}e_k,e_j)\rangle)\\
			=&\sqrt2(\langle i_{e_l}R^\triangledown,i_{e_{n+k}}R^\triangledown\rangle-\langle i_{e_k}R^\triangledown,i_{e_{n+l}}R^\triangledown\rangle)
		\end{align*}
		at $z_0$.\\
		(ii)\ Similarly, for any $1\le k<l\le n$, we have
			\begin{align*}
			&\sum_{i,j=1}^{2n}\langle R^\triangledown(e_i,e_j),R^\triangledown(D_{e_i}JV_{B_{kl}},e_j)\rangle\\
			=&\sum_{\alpha=1}^n\sum_{j=1}^{2n}(\langle R^\triangledown(e_\alpha,e_j),R^\triangledown(D_{e_\alpha}JV_{B_{kl}},e_j)\rangle+\langle R^\triangledown(e_{n+\alpha},e_j),R^\triangledown(D_{e_{n+\alpha}}JV_{B_{kl}},e_j)\rangle)\\
			=&-\frac1{\sqrt2}\sum_{\alpha=1}^n\sum_{j=}^{2n}(\langle R^\triangledown(e_\alpha,e_j),R^\triangledown(\delta_{\alpha l}e_k+\delta_{\alpha k}e_l,e_j)\rangle+\langle R^\triangledown(e_{n+\alpha},e_j),R^\triangledown(\delta_{\alpha l}e_{n+k}+\delta_{\alpha k}e_{n+l},e_j)\rangle)\\
			=&-\sqrt2(\langle i_{e_l}R^\triangledown,i_{e_k}R^\triangledown\rangle+\langle i_{e_{n+l}}R^\triangledown,i_{e_{n+k}}R^\triangledown\rangle).
		\end{align*}
		at $z_0$.\\
		(iii)\ For any $1\le t\le n$, by (\ref{DJV}) we have
		\begin{equation}\label{VC-sum}
			\begin{split}
				&\sum_{t=1}^n(\sum_{i,j=1}^{2n}\langle R^\triangledown(e_i,e_j),R^\triangledown(D_{e_i}JV_{C_t},e_j)\rangle)^2\\
				=&\sum_{t=1}^n(\sqrt{\frac{t+1}t}(|i_{e_t}R^\triangledown|^2+|i_{e_{n+t}}R^\triangledown|^2)+\frac1{\sqrt{t(t+1)}}\sum_{\gamma=t+1}^n(|i_{e_{\gamma}}R^\triangledown|^2+|i_{e_{n+\gamma}}R^\triangledown|^2))^2\\
				=&\sum_{t=1}^n(\frac{t+1}t(|i_{e_t}R^\triangledown|^2+|i_{e_{n+t}}R^\triangledown|^2)^2+\frac2t\sum_{\gamma=t+1}^n(|i_{e_t}R^\triangledown|^2+|i_{e_{n+t}}R^\triangledown|^2)(|i_{e_\gamma}R^\triangledown|^2+|i_{e_{n+\gamma}}R^\triangledown|^2)\\
				&+\frac1{t(t+1)}\sum_{\gamma,\mu=t+1}^n(|i_{e_\gamma}R^\triangledown|^2+|i_{e_{n+\gamma}}R^\triangledown|^2)(|i_{e_\mu}R^\triangledown|^2+|i_{e_{n+\mu}}R^\triangledown|^2)).
			\end{split}
		\end{equation}
		We decompose the summation index of the last term in the above equation into $\{t+1\le\gamma<\mu\le n\}, \{t+1\le\mu<\gamma\le n\}$ and $\{t+1\le\gamma=\mu\le n\}$. Changing the order of summation, we obtain
		\begin{align*}
			&\sum_{t=1}^n\frac1{t(t+1)}\sum_{\gamma,\mu=t+1}^n(|i_{e_\gamma}R^\triangledown|^2+|i_{e_{n+\gamma}}R^\triangledown|^2)(|i_{e_\mu}R^\triangledown|^2+|i_{e_{n+\mu}}R^\triangledown|^2)\\
			=&\sum_{t=1}^n(1-\frac1t)(|i_{e_t}R^\triangledown|^2+|i_{e_{n+t}}R^\triangledown|^2)^2+2\sum_{t=1}^n(1-\frac1t)\sum_{\gamma=t+1}^n(|i_{e_\gamma}R^\triangledown|^2+|i_{e_{n+\gamma}}R^\triangledown|^2)(|i_{e_t}R^\triangledown|^2+|i_{e_{n+t}}R^\triangledown|^2).
		\end{align*}
		By substituting the above equation into (\ref{VC-sum}), we have
		\begin{align*}
			&\sum_{t=1}^n(\sum_{i,j=1}^{2n}\langle R^\triangledown(e_i,e_j),R^\triangledown(D_{e_i}JV_{C_t},e_j)\rangle)^2\\
			=&2\sum_{t=1}^n((|i_{e_t}R^\triangledown|^2+|i_{e_{n+t}}R^\triangledown|^2)^2+\sum_{\gamma=t+1}^n(|i_{e_\gamma}R^\triangledown|^2+|i_{e_{n+\gamma}}R^\triangledown|^2)(|i_{e_t}R^\triangledown|^2+|i_{e_{n+t}}R^\triangledown|^2)).
		\end{align*}
		Note that
		\begin{align*}
			\sum_{t=1}^n|i_{e_t}R^\triangledown|^2+|i_{e_{n+t}}R^\triangledown|^2=\sum_{i=1}^{2n}|i_{e_i}R^\triangledown|^2=2|R^\triangledown|^2,
		\end{align*}
		we have
		\begin{align*}
			&\sum_{t=1}^n((|i_{e_t}R^\triangledown|^2+|i_{e_{n+t}}R^\triangledown|^2)^2+2\sum_{\gamma=t+1}^n(|i_{e_\gamma}R^\triangledown|^2+|i_{e_{n+\gamma}}R^\triangledown|^2)(|i_{e_t}R^\triangledown|^2+|i_{e_{n+t}}R^\triangledown|^2))\\
			=&\sum_{t,\gamma=1}^n(|i_{e_\gamma}R^\triangledown|^2+|i_{e_{n+\gamma}}R^\triangledown|^2)(|i_{e_t}R^\triangledown|^2+|i_{e_{n+t}}R^\triangledown|^2)\\
			=&(\sum_{t=1}^n(|i_{e_t}R^\triangledown|^2+|i_{e_{n+t}}R^\triangledown|^2)^2\\
			=&4|R^\triangledown|^4.
		\end{align*}
		Finally, since
		\begin{align*}
			\sum_{t=1}^n(|i_{e_t}R^\triangledown|^2+|i_{e_{n+t}}R^\triangledown|^2)^2=\frac12\sum_{i=1}^{2n}(|i_{e_i}R^\triangledown|^2+|i_{Je_i}R^\triangledown|^2)^2,
		\end{align*}
		we complete the proof of the lemma.
	\end{proof}
	The next lemma states the results at general $x\in\mathbb{C}P^n$.
	\begin{yl}\label{sum J4-3}
		Assume $\{V_1,...,V_q\}$ be an orthogonal basis of $\mathcal{K}$. Then we have
		\begin{align*}
			\sum_{k=1}^q(\sum_{i,j=1}^{2n}\langle R^\triangledown(e_i,e_j),R^\triangledown(D_{e_i}JV_k,e_j)\rangle)^2=\sum_{i,j=1}^{2n}(\langle i_{e_i}R^\triangledown,i_{e_j}R^\triangledown\rangle^2+\langle i_{e_i}R^\triangledown,i_{e_j}R^\triangledown\rangle\langle i_{Je_i}R^\triangledown,i_{Je_j}R^\triangledown\rangle)+4|R^\triangledown|^4.
		\end{align*}
	\end{yl}
	\begin{proof}
		Due to the symmetry of $\mathbb{C}P^n$, we only need to prove the lemma at $z_0$. Assume $\{e_i\}$ is defined in (\ref{ealpha}). By lemma \ref{sum J4-2} and $Je_\alpha=e_{n+\alpha}, Je_{n+\alpha}=-e_\alpha$ for any $1\le\alpha\le n$, we finish the proof.
	\end{proof}
	Based on lemma \ref{sum J4-1}, lemma \ref{sum J4-2} and lemma \ref{sum J4-3}, we can complete the calculation of $\sum_{k=1}^qJ_V^4$.
	\begin{yl}\label{sum J4}
		Assume $\{V_1,...,V_q\}$ be an orthogonal basis of $\mathcal{K}$. For any $x\in\mathbb{C}P^n$, we have
		\begin{align*}
			\sum_{k=1}^qJ_{V_k}^4(x)=&F'(2|R^\triangledown|^2+\sum_{i,j=1}^{2n}\langle R^\triangledown(e_i,e_j),R^\triangledown(Je_i,Je_j)\rangle+|\sum_{i=1}^{2n}R^\triangledown(e_i,Je_i)|^2)\\
			+&F''(4|R^\triangledown|^4+\sum_{i,j=1}^{2n}(\langle i_{e_i}R^\triangledown,i_{e_j}R^\triangledown\rangle^2+\langle i_{e_i}R^\triangledown,i_{e_j}R^\triangledown\rangle\langle i_{Je_i}R^\triangledown,i_{Je_j}R^\triangledown\rangle)).
		\end{align*}
	\end{yl}
	We still assume $\{V_1,...,V_q\}$ be an orthogonal basis of $\mathcal{K}$ and let $B_{V_1},...,B_{V_q}$ be the variations. Lemma \ref{2nd var}, lemma \ref{sum J123} and lemma \ref{sum J4} show that
	\begin{align}\label{sum B_V}
		\sum_{k=1}^q\mathscr{L}_F^\triangledown(B_{V_k})=\int_{\mathbb{C}P^n}\sum_{k=1}^qJ_{V_k}^4(x)dV,
	\end{align}
	where $\sum_{k=1}^qJ_{V_k}^4(x)$ is given by lemma \ref{sum J4}.
	\subsection{The estimation of the second variations under specific variation fields}
	We will prove the main theorem in this section. Proceeding from (\ref{sum B_V}), we need to estimate $\sum_{k=1}^qJ_{V_k}^4$.
	\begin{yl}\label{estimate}
		For any $x\in\mathbb{C}P^n$, we assume $\{e_1,...,e_{2n}\}$ be an orthogonal basis of $T_x\mathbb{C}P^n$. Then we have
		\begin{align*}
			(i)\ &2|R^\triangledown|^2+\sum_{i,j=1}^{2n}\langle R^\triangledown(e_i,e_j),R^\triangledown(Je_i,Je_j)\rangle+|\sum_{i=1}^{2n}R^\triangledown(e_i,Je_i)|^2\le(4+4n)|R^\triangledown|^2,\\
			(ii)\  &4|R^\triangledown|^4+\sum_{i,j=1}^{2n}(\langle i_{e_i}R^\triangledown,i_{e_j}R^\triangledown\rangle^2+\langle i_{e_i}R^\triangledown,i_{e_j}R^\triangledown\rangle\langle i_{Je_i}R^\triangledown,i_{Je_j}R^\triangledown\rangle)\ge(4+\frac8n)|R^\triangledown|^4
		\end{align*}
		at $x$, where the two unequal signs both take equal signs if and only if there exists $\sigma\in\mathfrak{so}(E_x)$ such that
		\begin{align}\label{condition for =}
			R^\triangledown(X,Y)=g(X,JY)\sigma
		\end{align}
		for any $X,Y\in T_x\mathbb{C}P^n$.
	\end{yl}
	\begin{proof}
		(i)\ By Young's inequality, we have
		\begin{align*}
			&\sum_{i,j=1}^{2n}\langle R^\triangledown(e_i,e_j),R^\triangledown(Je_i,Je_j)\rangle\le\sum_{i,j=1}^{2n}|R^\triangledown(e_i,e_j)||R^\triangledown(Je_i,Je_j)|\\
			\le&\frac12\sum_{i,j=1}^{2n}(|R^\triangledown(e_i,e_j)|^2+|R^\triangledown(Je_i,Je_j)|^2)\\
			=&2|R^\triangledown|^2,
		\end{align*}
		where the equal sign is taken if and only if
		\begin{align*}
			R^\triangledown(e_i,e_j)=R^\triangledown(Je_i,Je_j)
		\end{align*}
		for any $1\le i,j\le 2n$. By Cauchy's inequality, we have
		\begin{align*}
			|\sum_{i=1}^{2n}R^\triangledown(e_i,Je_i)|^2\le(\sum_{i=1}^{2n}|R^\triangledown(e_i,Je_i)|)^2\le 2n\sum_{i=1}^{2n}|R^\triangledown(e_i,Je_i)|^2\le4n|R^\triangledown|^2,
		\end{align*}
		where the necessary and sufficient conditions for taking equal signs of the three inequality are as follows:
		\begin{align*}
			(1)\ &\textrm{There exists }\sigma\in\mathfrak{so}(E_x)\textrm{ and }\mu_i\ge0,\ i=1,...,2n,\textrm{ such that }R^\triangledown(e_i,Je_i)=\mu_i\sigma \textrm{ for any }1\le i\le 2n,\\
			(2)\ &|R^\triangledown(e_i,Je_i)|=|R^\triangledown(e_j,Je_j)|\textrm{ for any }1\le i,j\le 2n,\\
			(3)\ &\textrm{There exists }\sigma_{ij}\in\mathfrak{so}(E_x),\ 1\le i,j\le 2n,\textrm{ such that }R^\triangledown(e_i,e_j)=g(e_i,Je_j)\sigma_{ij}.
		\end{align*}
		Hence these three inequalities are all equal if and only if (\ref{condition for =}) holds.\\
		(ii)\ Note that
		\begin{align*}
			\sum_{i,j=1}^{2n}(\langle i_{e_i}R^\triangledown,i_{e_j}R^\triangledown\rangle^2+\langle i_{e_i}R^\triangledown,i_{e_j}R^\triangledown\rangle\langle i_{Je_i}R^\triangledown,i_{Je_j}R^\triangledown\rangle)=\frac12\sum_{i,j=1}^{2n}(\langle i_{e_i}R^\triangledown,i_{e_j}R^\triangledown\rangle+\langle i_{Je_i}R^\triangledown,i_{Je_j}R^\triangledown\rangle)^2.
		\end{align*}
		By Cauchy's inequality, we have
		\begin{align*}
			&2n\sum_{i,j=1}^{2n}(\langle i_{e_i}R^\triangledown,i_{e_j}R^\triangledown\rangle+\langle i_{Je_i}R^\triangledown,i_{Je_j}R^\triangledown\rangle)^2\\
			\ge&(\sum_{i,j=1}^{2n}|\langle i_{e_i}R^\triangledown,i_{e_j}R^\triangledown\rangle+\langle i_{Je_i}R^\triangledown,i_{Je_j}R^\triangledown\rangle|)^2\\
			\ge&(\sum_{i=1}^{2n}|i_{e_i}R^\triangledown|^2+|i_{Je_i}R^\triangledown|^2)^2\\
			=&16|R^\triangledown|^4.
		\end{align*}
		The above inequality is an equality if (\ref{condition for =}) holds.
	\end{proof}
	By combining lemma \ref{sum J4}, lemma \ref{estimate}, and (\ref{sum B_V}), we can complete the proof of theorem \ref{FYM-stability}.
	\begin{remark}\label{rmk:eigenfunction, projection}
		For any $V\in\mathcal{K}$, $JV$ is the gradient of a eigenfunction corresponding to the first eigenvalue of the Laplace-Beltrami operator and such eigenfunction has the form
		\begin{align*}
			f_Q(z)=\frac{\sum_{0\le i,j\le n}q_{ij}z_i\bar z_j}{|z|^2},
		\end{align*}
		where $Q=(q_{ij})_{0\le i,j\le n}$ is a Hermitian metric and $tr(Q)=0$. Adopting the approach of constructing variations through gradients of eigenfunctions corresponding to the first eigenvalue of the Laplacian on a manifold coincides with the selection method on spheres (e.g. see \cite{JL}). Moreover, Santiago \cite{SAN} gives the isometric embedding $\iota:\mathbb{C}P^n\hookrightarrow\mathbb{R}^N$. Then $JV$ is the projection of a constant vector field on $\mathbb{R}^N$, that is, 
		\begin{align*}
			\iota_\ast JV=\sum_{j=1}^Na_j\frac{\partial}{\partial y^j}\mid_{\iota(\mathbb{C}P^n)}
		\end{align*}
		for some $a_1,...,a_N\in\mathbb{R}$. Thus our choice of the variation is consist with that by using the second fundamental form (e.g. see \cite{KK}).
	\end{remark}
	\subsection{A gap theorem}
	In this section, we give a gap theorem of $F$-Yang-Mills functional on $\mathbb{C}P^n$.
	\begin{thm}
		Assume $\nabla$ is a $F$-Yang-Mills connection on $\mathbb{C}P^n$ $(n\ge2)$ and $F''\ge0$. If $$\|R^\triangledown\|_{L^\infty}\le\frac{(2n-1)\sqrt{2n(2n-1)}}{8(n-1)},$$ then $R^\triangledown\equiv0$.
	\end{thm}
	\begin{proof}
		For any $x\in\mathbb{C}P^n$, let $\{e_1,...,e_{2n}\}$ be an orthogonal basis of $T\mathbb{C}P^n$ such that $De_j(x)=0$. Then we have
 	    \begin{equation}\label{lap F}
		    \begin{split}
			    \Delta F(\frac12|R^\triangledown|^2)=-\sum_je_je_j(F)=-\frac14F''grad(|R^\triangledown|^2)^2-F'|\nabla R^\triangledown|^2+F'\langle\nabla^\ast\nabla R^\triangledown,R^\triangledown\rangle,
	  	    \end{split}
	    \end{equation}
        $d^\triangledown R^\triangledown=0$ and the Bochner-Weizenb{\"o}ck formula (\ref{Bochner}) shows that
	    \begin{align*}
		    \nabla^\ast\nabla R^\triangledown=d^\triangledown\delta^\triangledown R^\triangledown-R^\triangledown\circ(Ric\wedge I+2R)-\mathfrak{R}^\triangledown(R^\triangledown).
	    \end{align*}
	    The $F$-Yang-Mills equation implies $\int_{\mathbb{C}P^n}F'\langle d^\triangledown\delta^\triangledown R^\triangledown,R^\triangledown\rangle dV=0.$ Integral (\ref{lap F}) and note that $\int_{\mathbb{C}P^n}\Delta FdV=0$, we have
	    \begin{align}\label{gap}
	    	\int_{\mathbb{C}P^n}F'\langle R^\triangledown\circ(Ric\wedge I+2R)+\mathfrak{R}^\triangledown(R^\triangledown),R^\triangledown\rangle dV\le0.
	    \end{align}
	    By direct calculation, we have $\langle R^\triangledown\circ(Ric\wedge I),R^\triangledown\rangle =(2n+2)|R^\triangledown|^2$ and
	    \begin{align*}
	    	\langle R^\triangledown\circ2R,R^\triangledown\rangle=&-|R^\triangledown|^2-\frac12\sum_{i,j}\langle R^\triangledown(Je_i,Je_j),R^\triangledown(e_i,e_j)\rangle-\frac12\sum_{i,j}\langle R^\triangledown(e_i,Je_i),R^\triangledown(e_j,Je_j)\rangle\\
	    	\ge&-|R^\triangledown|^2-\frac14\sum_{i,j}(|R^\triangledown(Je_i,Je_j)|^2+|R^\triangledown(e_i,e_j)|^2)-\frac14\sum_{i,j}(|R^\triangledown(e_i,Je_i)|^2+|R^\triangledown(e_j,Je_j)|^2)\\
	    	\ge&-3|R^\triangledown|^2.
	    \end{align*}
	    By proposition 5.6 in \cite{JL}, we have $\langle \mathfrak{R}^\triangledown(R^\triangledown),R^\triangledown\rangle\ge-\frac{8(n-1)}{\sqrt{2n(2n-1)}}|R^\triangledown|^3$. Put them into (\ref{gap}) and we have
	    \begin{align}
	    	\int_{\mathbb{C}P^n}(2n-1-\frac{8(n-1)}{\sqrt{2n(2n-1)}}|R^\triangledown|)|R^\triangledown|^2dV\le0.
	    \end{align}
	    Then we finish the proof.
	\end{proof}


\begin{thebibliography}{100}		
		\bibitem{Xin1} YuanLong, X., ``Some results on stable harmonic maps," \emph{Duke Math. J.} $\mathbf{47(3)}$, 609-613 (1980).
		\bibitem{JL} Bourguignon, JP., and Lawson, H.B., ``Stability and isolation phenomena for Yang-Mills fields," \emph{Commun.Math. Phys.} $\mathbf{79}$, 189–230 (1981). 
		\bibitem{Laquer} Laquer, H.T., ``Stability properties of the Yang-Mills functional near the canonical connection,"  \emph{Mich. Math. J.} $\mathbf{31}$, 139-159 (1984).
		\bibitem{COT} Kobayashi, S., Ohnita, Y. and Takeuchi, M., ``On instability of Yang-Mills connections," \emph{Math. Z.} $\mathbf{193}$, 165–189 (1986). 
		\bibitem{Parker} Parker, T., ``Conformal fields and stability," \emph{Math. Z. } $\mathbf{185}$, 305-319 (1984).
		\bibitem{HTY} Hong, M. C., Tian, G., Yin, H.,: ``The Yang–Mills-flow in vector bundles over four manifolds and its applications." \emph{Comment. Math. Helv.} $\mathbf{90}$,75–120 (2015).
		\bibitem{FH} Fumiaki, M., Hajime, U., ``On exponential Yang-Mills connections”, \emph{J. Geom. Phys.}, $\mathbf{17}$ , 73–89 (1995).
		\bibitem{KK} Baba, K. and Shintani, K., ``A Simons type condition for instability of $F$-Yang-Mills connections," \emph{arXiv:} $\mathbf{2301.04291}$, (2023).
		\bibitem{YO} Yoshihiro, O., ``Stability of harmonic maps and standard minimal immersions", \emph{Tohoku Math. Journ.} $\mathbf{38}$, 259-267 (1986).
		\bibitem{Cheng} DaRong, C., ``Instability of solutions to the Ginzburg–Landau
		equation on $S^n$ and $\mathbb{C}P^n$," \emph{arXiv:} $\mathbf{1911.04097}$, (2019).
		\bibitem{Jost} Jost, J. ``Riemannian Geometry and Geometric Analysis," Berlin: Springer, (2008).
		\bibitem{SK} Kobayashi, S., and Nomizu, K., ``Foundations of differential geometry. Vol. II," Interscience Publication, Wiley, New York, (1969).
		\bibitem{SAN} Simanca, S.R., ``Canonical isometric embeddings of projective spaces into spheres," \emph{Bol. Soc. Mat. Mex.} $\mathbf{26}$, 757–763 (2020).
	\end{thebibliography}
\end{document}